\documentclass[12pt,reqno]{article}

\usepackage[usenames]{color}
\usepackage[colorlinks=true,
linkcolor=webgreen, filecolor=webbrown,
citecolor=webgreen]{hyperref}

\definecolor{webgreen}{rgb}{0,.5,0}
\definecolor{webbrown}{rgb}{.6,0,0}

\usepackage{amssymb}
\usepackage{graphicx}
\usepackage{amscd}
\usepackage{lscape}
\usepackage{tikz}
\usepackage{tikz-cd}
\usepackage{pgfplots}

\usetikzlibrary{matrix}
\usetikzlibrary{fit,shapes}
\usetikzlibrary{positioning, calc}
\tikzset{circle node/.style = {circle,inner sep=1pt,draw, fill=white},
        X node/.style = {fill=white, inner sep=1pt},
        dot node/.style = {circle, draw, inner sep=5pt}
        }
\usepackage{tkz-fct}

\usepackage{amsthm}
\newtheorem{theorem}{Theorem}
\newtheorem{lemma}[theorem]{Lemma}
\newtheorem{proposition}[theorem]{Proposition}
\newtheorem{corollary}[theorem]{Corollary}

\theoremstyle{definition}

\newtheorem{example}[theorem]{Example}

\usepackage{float}

\usepackage{graphics,amsmath}
\usepackage{amsfonts}
\usepackage{latexsym}
\usepackage{epsf}

\newcommand{\seqnum}[1]{\href{http://oeis.org/#1}{\underline{#1}}}

\begin{document}

\begin{center}
\vskip 1cm{\LARGE\bf Riordan arrays and Jacobi and Thron continued fractions} \vskip 1cm \large
Paul Barry\\
School of Science\\
Waterford Institute of Technology\\
Ireland\\
\href{mailto:pbarry@wit.ie}{\tt pbarry@wit.ie}
\end{center}
\vskip .2 in

\begin{abstract} We show that certain Riordan arrays have generating functions that can be expressed as continued fractions of Jacobi and Thron type. We investigate the inverses of such arrays, which in certain circumstances can also have generating functions representable as continued fractions. Links to orthogonal polynomial moment sequences, and to Laurent biorthogonal polynomials are developed. We show that certain Riordan group involutions can be defined by continued fractions. We also show how simple transformations of the Jacobi continued fractions can lead to exponential Riordan arrays. Finally, by way of contrast, we look at the case of some non Riordan arrays that are of combinatorial significance, including the Narayana numbers. \end{abstract}

\section{Introduction}

A Riordan array \cite{Book, SGWW} can be defined as a couple $(g(x), f(x))$ of (formal) power series where
$$g(x)=g_0+g_1 x + g_2 x^2 +\cdots,$$
$$f(x)=f_1x+f_2x^2+f_3x^3+\cdots,$$ with $g_0 \ne 0$, $f_0=0$ and $f_1 \ne 0$.
To each such pair we can associate the matrix whose $(n,k)$-th element $a_{n,k}$ is given by
$$a_{n,k}=[x^n] g(x)f(x)^k$$ where $[x^n]$ is the functional on the space of power series that extracts the coefficient of $x^n$. Thus $[x^n]f(x)=f_n$.
We can define a product for pairs of Riordan arrays $(g(x), f(x))$ and $(u(x), v(x))$ as follows.
$$(g(x), f(x)) \cdot (u(x), v(x))= (g(x) u(f(x)), v(f(x)).$$
In terms of the corresponding matrices, this is realised as the usual matrix product. We can define an inverse for this product, given as follows.
$$(g(x), f(x))^{-1}=\left(\frac{1}{g(\bar{f}(x))}, \bar{f}(x)\right)$$ where
$\bar{f}(x)$ is the compositional inverse of $f(x)$. That is, it is the solution $u(x)$ of the equation $f(u)=x$ which satisfies $u(0)=0$. The existence of such an inverse is assured by the conditions $f_0=0$ and $f_1 \ne 0$.
The matrix corresponding to $(g(x), f(x))^{-1}$ is then the inverse of the matrix corresponding to $(g(x), f(x))$. The identity element is given by $(1,x)$, and with this the set of Riordan arrays becomes a group, called the Riordan group. Where no confusion can arise, we often do not distinguish between the Riordan array $(g(x), f(x))$ and its matrix representation.

The Riordan group has an action on the ring of power series as follows.
$$(g(x), f(x))\cdot h(x)= g(x)h(f(x)).$$
This is often called \emph{the fundamental theorem of Riordan arrays}. In matrix terms, the left hand side corresponds to multiplying the column vector $(h_0, h_1, h_2,\ldots)^T$ by the matrix representing $(g(x), f(x))$, with the result being the column vector whose elements are obtained by expanding the power series $g(x)h(f(x))$.

The bivariate generating function $G(x,y)=\sum_{n,k} a_{n,k}x^n y^k$ of $(g(x), f(x))$ corresponds to the result of the action
$$ (g(x), f(x)) \cdot \frac{1}{1-yx}=\frac{g(x)}{1-y f(x)}.$$
That is, we have
$$a_{n,k}=[x^n][y^k] \frac{g(x)}{1-y f(x)}.$$
In fact, we have
\begin{align*} [x^n][y^k] \frac{g(x)}{1-y f(x)}&= [x^n][y^k] g(x)\sum{i=0}^{\infty} y^i f(x)^i\\
&=[x^n] g(x) [y^k]\sum{i=0}^{\infty} y^i f(x)^i\\
&=[x^n] g(x) f(x)^k.\end{align*}

The following simple result will be used in the sequel.
\begin{lemma} Let $G(x,y)=\frac{g(x)}{1-yf(x)}$ be the generating function of the Riordan array $(g(x),f(x))$.
Then $G(x,0)$ is the generating function of the first column of the array, and $G(x,1)$ is the generating function of the row sums of the array.
\end{lemma}
\begin{proof} We immediately have $G(x,0)=g(x)$. We also have
$$G(x,1)=\frac{g(x)}{1-f(x)}=(g(x), f(x))\cdot \frac{1}{1-x}.$$ Alternatively we have
\begin{align*}[x^n]G(x,1)&=[x^n] \frac{g(x)}{1-f(x)}\\
&=[x^n]g(x)\sum_{i=0}^{\infty} f(x)^i\\
&=\sum_{i=0}^n[x^n] g(x)f(x)^i\\
&=\sum_{i=0}^n a_{n,i}.\end{align*}
\end{proof}
\begin{corollary} If $G(x,y)$ is the generating function of a Riordan array $(g(x), f(x))$, then
$$g(x)=G(x,1),\quad f(x)=1-\frac{G(x,0)}{G(x,1)}.$$
\end{corollary}
\begin{proof} We have $\frac{g(x)}{1-f(x)}=\frac{G(x,0)}{1-f(x)}=G(x,1)$. Solving for $f(x)$ gives
$$f(x)=1-\frac{G(x,0)}{G(x,1)}.$$
\end{proof}

We note that if $G(x,y)$ is the generating function of a number triangle, then $G(x,x)$ is the generating function of its diagonal sums.

\begin{example} Pascal's triangle $\left(\binom{n}{k}\right)_{0 \le n,k \le \infty}$, regarded as a lower triangular matrix, is given by the Riordan array $\left(\frac{1}{1-x}, \frac{x}{1-x}\right)$. Its (bivariate) generating function is thus given by
$$G(x,y)=\frac{\frac{1}{1-x}}{1-y\frac{x}{1-x}}=\frac{1}{1-x-xy}.$$
We have $G(x,0)=\frac{1}{1-x}, G(x,1)=\frac{1}{1-2x}$, and $G(x,x)=\frac{1}{1-x-x^2}$.
Thus the first column consists of $1$'s, the row sums are given by $2^n$, and the diagonal sums are the Fibonacci numbers $F_{n+1}$ \seqnum{A000045}. This matrix is also known as the binomial matrix \seqnum{A007318}.
\end{example}

Pascal's triangle is an example of a Bell array, which is a Riordan array of the form $(g(x), xg(x))$. Such arrays form a subgroup of the Riordan group.

We now ask ourselves the question: what can we say about a generating function given by a continued fraction, such as
$$\cfrac{1}{1-x-\cfrac{xy}{1-x-\cfrac{x}{1-x-\cfrac{x}{1-x- \cdots}}}}\quad?$$
In fact, this continued fraction is the generating function of the Riordan array
$$\left(\frac{1}{1-x}, \frac{1-x-\sqrt{1-6x+x^2}}{2(1-x)}\right).$$ This array begins
$$\left(
\begin{array}{ccccc}
 1 & 0 & 0 & 0 & 0 \\
 1 & 1 & 0 & 0 & 0 \\
 1 & 4 & 1 & 0 & 0 \\
 1 & 13 & 7 & 1 & 0 \\
 1 & 44 & 34 & 10 & 1 \\
\end{array}
\right).$$

The type of continued fraction in this example is called a \emph{Thron continued fraction}, or T-type continued faction. Thron continued fractions are associated with Schr\"oder paths. The above continued fraction represents Schr\"oder paths where the rise steps at level $0$ have a weight of $y$, which may be expressed by saying they have $y$ colors. For instance, we have

$$\left(
\begin{array}{ccccc}
 1 & 0 & 0 & 0 & 0 \\
 1 & 1 & 0 & 0 & 0 \\
 1 & 4 & 1 & 0 & 0 \\
 1 & 13 & 7 & 1 & 0 \\
 1 & 44 & 34 & 10 & 1 \\
\end{array}
\right)\left(\begin{array}{c}
1\\
3\\
9\\
27\\
81\\ \end{array}\right)=
\left(\begin{array}{c}
1\\
4\\
22\\
130\\
790\\ \end{array}\right).$$

This means that the sequence $1,4,22,130,790,\ldots$ \seqnum{A155862} counts Schr\"oder paths (of length $2n$) whose rise step at level $0$ has $3$ colors.

A Riordan array of the form $\left(\frac{1+\alpha x+ \beta x^2}{1+ax+bx^2}, \frac{x}{1+a x + bx^2}\right)$ is the coefficient array of the family of orthogonal polynomials $P_n(x)$ which satisfies the recurrence relation
$$P_n(x)=(x-a)P_{n-1}(x)-b P_{n-2}(x).$$ The production matrix of the inverse of this matrix begins
 $$\left(
\begin{array}{cccccc}
 a-\alpha & 1 & 0 & 0 & 0 & 0 \\
 b-\beta & a & 1 & 0 & 0 & 0 \\
 0 & b & a & 1 & 0 & 0 \\
 0 & 0 & b & a & 1 & 0 \\
 0 & 0 & 0 & b & a & 1 \\
 0 & 0 & 0 & 0 & b & a \\
\end{array}
\right).$$ Here, if $M$ is an invertible matrix, then the production matrix of $M$  is given by $M^{-1} \overline{M}$, where $\overline{M}$ is the matrix $M$ with its top row removed. This tri-diagonal matrix is associated with the Jacobi continued fraction
$$\cfrac{1}{1-(a-\alpha)x-\cfrac{(b-\beta)x^2}{1-ax-\cfrac{bx^2}{1-ax-\cdots}}},$$ which is the generating function of the first column of the moment matrix $\left(\frac{1+\alpha x+ \beta x^2}{1+ax+bx^2}, \frac{x}{1+a x + bx^2}\right)^{-1}$ \cite{OP}.

Where known, integer sequences appearing in this note will be referenced by their entry number $Annnnnn$ in the Neil Sloane's On-Line Encyclopedia of Integer Sequences \cite{SL1, SLII}. Numerous examples of Riordan arrays and of sequences whose generating functions are given by continued fractions may be found therein. 
In the next section, we review the topics of lattice paths and continued paths.

\section{Lattice paths - a review}
A \emph{Dyck path} (see Figure \ref{Dyck}) is a path in the first quadrant that begins at the origin $(0,0)$, ends at $(2n,0)$, and consist of steps $(1,1)$ (North-East), called \emph{rises} or \emph{up steps}, and $(1,-1)$ (South-East), called \emph{falls} or \emph{down steps}. We refer to $n$ as the \emph{semi-length} of the path. Dyck paths of semi-length $n$ are sometimes referred to as Dyck $n$-paths. A \emph{peak} of a Dyck path is the joint node formed by a rise step immediately followed by a fall step. The \emph{height} of a peak is the $y$-coordinate of this node.

\begin{figure}
\begin{center}
\begin{tikzpicture}[set style={{help lines}+=[dashed]}]
\draw (0,0)--(1,1)--(2,2)--(3,3)--(4,2)--(5,3)--(6,2)--(7,1)--(8,2)--(9,1)--(10,0);
\draw[style=help lines] (0,0) grid (10,3);
\end{tikzpicture}
\caption{A Dyck path}\label{Dyck}
\end{center}
\end{figure}
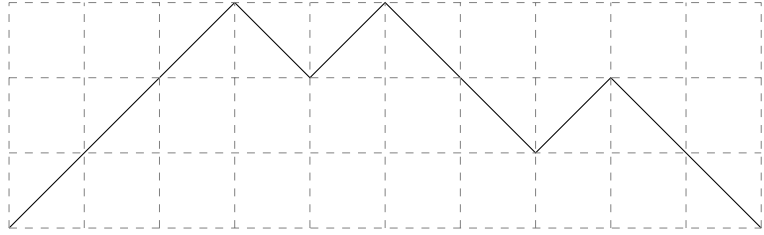

A \emph{Motzkin path} (see Figure \ref{Motzkin}) is a path in the first quadrant which begins at the origin $(0,0)$, ends at $(n,0)$, and consists of steps $(1,1)$ (North-East), called \emph{rises}, and (1,-1) (South-East), called \emph{falls}, and steps $(1,0)$ (East) called \emph{horizontals}. A partial Motzkin path that starts from $(0,0)$ and ends at the point $(n,k)$ (not necessarily on the $x$-axis) is called a \emph{left factor} of a Motzkin path.
\begin{figure}
\begin{center}
\begin{tikzpicture}[set style={{help lines}+=[dashed]}]
\draw plot coordinates {(0,0) (1,1) (2,0) (3,1) (4,1) (5,2) (6,2) (7,1) (8,0) (9,1) (10,0) };
\draw[style=help lines] (0,0) grid (10,3);
\end{tikzpicture}
\end{center}
\caption{A Motzkin path}\label{Motzkin}
\end{figure}
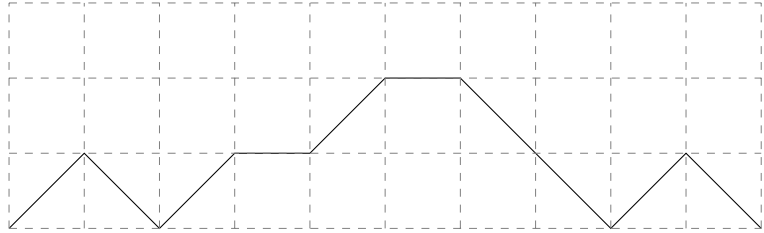

\begin{figure}
\begin{center}
\begin{tikzpicture}[set style={{help lines}+=[dashed]}]
\draw plot coordinates {(0,0) (1,1) (2,1) (3,2) (5,0) (6,0)  (7,1) (8,1) (9,0) };
\node at (0.2, 0.5) {$a_0$};
\node at (1.4, 1.2) {$c_1$};
\node at (2.2, 1.5) {$a_1$};
\node at (3.7, 1.5) {$b_2$};
\node at (4.7, 0.5) {$b_1$};
\node at (5.4, 0.2) {$c_0$};
\node at (6.2, 0.5) {$a_0$};
\node at (7.4, 1.2) {$c_1$};
\node at (8.7, 0.5) {$b_1$};
\draw[style=help lines] (0,0) grid (12,3);
\draw[style=help lines] (7,4) grid (9,6);
\draw (7,5) -- (8,4); 
\node at (5.8,5) {level $k$}; 
\node at (8.5, 4.2) {$b_k$};
\node at (8.5, 5) {$c_k$};
\node at (8.5, 5.7) {$a_k$};
\node at (7.8, 6.3) {valuation $\nu$};
\draw (7,5) -- (8,5); 
\draw (7,5) -- (8,6); 
\end{tikzpicture}
\end{center}
\caption{A weighted Motzkin path $\nu(\omega)=a_0^2a_1b_1^2b_2c_0c_1^2$}\label{Motzkin_ext}
\end{figure}
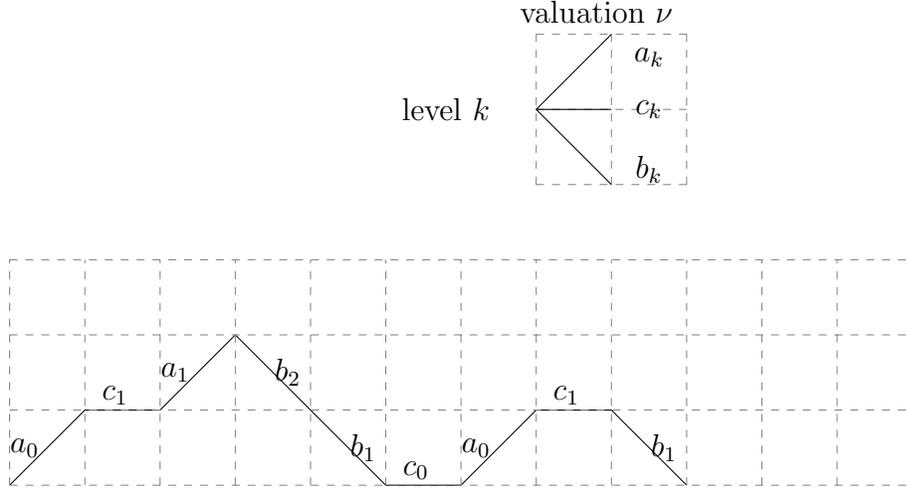

Finally a \emph{Schr\"oder path} (see Figure \ref{Schroeder}) is a path in the first quadrant which begins at the origin $(0,0)$, ends at $(2n,0)$, and consists of steps $(1,1)$ (North-East), called \emph{rises}, and $(1,-1)$ (South-East) called \emph{falls}, and steps $(2,0)$ (East) called \emph{horizontals} or \emph{level steps}. Note that a Dyck path is a Schr\"oder path with no horizontals.

\begin{example} The Riordan array
$$\left(\frac{1-2x}{1-x}, \frac{x(1-2x)}{1-x}\right)^{-1}=$$
$$\left(\frac{1+x-\sqrt{1-6x+x^2}}{4x}, \frac{1+x-\sqrt{1-6x+x^2}}{4}\right)$$
counts the number of Schr\"oder paths of length $2n$ that have $k$ peaks at height $1$. This member of the Bell subgroup of the Riordan group begins
$$\left(
\begin{array}{cccccc}
 1 & 0 & 0 & 0 & 0 & 0 \\
 1 & 1 & 0 & 0 & 0 & 0 \\
 3 & 2 & 1 & 0 & 0 & 0 \\
 11 & 7 & 3 & 1 & 0 & 0 \\
 45 & 28 & 12 & 4 & 1 & 0 \\
 197 & 121 & 52 & 18 & 5 & 1 \\
\end{array}
\right),$$ with a production matrix that begins
$$\left(
\begin{array}{cccccc}
 1 & 1 & 0 & 0 & 0 & 0 \\
 2 & 1 & 1 & 0 & 0 & 0 \\
 4 & 2 & 1 & 1 & 0 & 0 \\
 8 & 4 & 2 & 1 & 1 & 0 \\
 16 & 8 & 4 & 2 & 1 & 1 \\
 32 & 16 & 8 & 4 & 2 & 1 \\
\end{array}
\right).$$
\end{example}

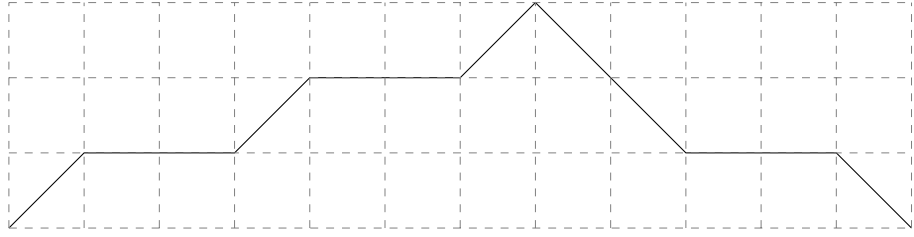
\begin{figure}
\begin{center}
\begin{tikzpicture}[set style={{help lines}+=[dashed]}]
\draw (0,0)--(1,1)--(3,1)--(4,2)--(6,2)--(7,3)--(9,1)--(11,1)--(12,0);
\draw[style=help lines] (0,0) grid (12,3);
\end{tikzpicture}
\caption{A Schr\"oder path}\label{Schroeder}
\end{center}
\end{figure}

Dyck paths of semi-length $n$ are counted by the Catalan numbers $C_n=\frac{1}{n+1}\binom{2n}{n}$ \seqnum{A000108}, since they both obey the same recurrence relation. The generating function of the Catalan numbers, $c(x)=\frac{1-\sqrt{1-4x}}{2x}$, can be expressed as a continued fraction as follows.
$$c(x)=\cfrac{1}{1-\cfrac{x}{1-\cfrac{x}{1-\cdots}}}.$$
We can see this by solving the equation
$$u=\frac{1}{1-x u},$$
equivalent to the continued fraction expression. We obtain
$$u=\frac{1-\sqrt{1-4x}}{2x} \quad \text{or}\quad u=\frac{1+\sqrt{1-4x}}{2x}.$$
The first form satisfies $u(0)=0$, so this is the solution we seek.

Continued fractions of the form
$$\cfrac{1}{1-\cfrac{a x}{1-\cfrac{bx}{1-\cfrac{c x}{1-\cdots}}}}$$ are known as \emph{Stieltjes continued fractions}.
Flajolet \cite{Flajolet} showed that
$$\cfrac{1}{1-\cfrac{\alpha_1 x}{1-\cfrac{\alpha_2 x}{1-\cdots}}}=\sum_{n=0}^{\infty}S_n(\alpha_1,\ldots,\alpha_n)x^n,$$
where $S_n(\alpha_1,\ldots,\alpha_n)$ is the generating function for Dyck paths of semi-length $n$ in which each fall starting at height $i$ has weight $\alpha_i$ (rises have weight $1$).

Motzkin paths of length $n$ are counted by the Motzkin numbers $M_n$, where
$$M_n=\sum_{k=0}^{\lfloor \frac{n}{2} \rfloor} \binom{n}{2k} C_k.$$
The generating function $m(x)$ of the Motzkin numbers is given by
$$m(x)=\frac{1-x-\sqrt{1-2x-3x^2}}{2x^2}.$$
We have
$$m(x)=\cfrac{1}{1-x-\cfrac{x^2}{1-x-\cfrac{x^2}{1-x-\cdots}}}.$$
Continued fractions of the form
$$\cfrac{1}{1-\alpha_0 x-\cfrac{\beta_1 x^2}{1-\alpha _1 x-\cfrac{\beta_2 x^2}{1-\alpha_2 x-\cdots}}}$$
are called \emph{Jacobi continued fractions}.
Flajolet \cite{Flajolet} has shown that
$$\cfrac{1}{1-\alpha_0 x-\cfrac{\beta_1 x^2}{1-\alpha _1 x-\cfrac{\beta_2 x^2}{1-\alpha_2 x-\cdots}}}=\sum_{n=0}^{\infty} J_n(\mathbf{\alpha}, \mathbf{\beta})x^n,$$
where $J_n(\mathbf{\alpha}, \mathbf{\beta})$ is the generating function for Motzkin paths of length $n$ in which each fall at height $i$ has weight $\alpha_i$,  each horizontal step at height $i$ has weight $\beta_i$, and rises have weight $1$. More generally, Flajolet showed the following. If we denote by $|\omega|$ the length of a Motzkin path $\omega$, and by $\nu(\omega)$ the weight of the Motzkin path $\omega$ (see Figure \ref{Motzkin_ext}), then we have ``Flajolet's fundamental lemma of Jacobi continued fractions'' \cite{Viennot}
$$\sum_{\omega:\text{Motzkin paths}}\nu(\omega)x^{|\omega|}=\cfrac{1}{1-c_0 x- \cfrac{a_0 b_1x^2}{1-c_1x-\cfrac{a_1b_2x^2}{1-c_2x-\cfrac{a_2b_3x^2}{1-c_3x-\cdots}}}}.$$ 
Thus the coefficient of $x^2$ is a product of weights: in this note, when we talk of colored rises, for instance, we attribute a weight of $1$ to the corresponding fall, and vice versa.

Schr\"oder paths of length $2n$ are counted by the large Schr\"oder numbers \seqnum{A006318}
$$S_n=\sum_{k=0}^n \binom{n+k}{2k} C_k.$$
The sequence $S_n$ has its generating function $s(x)$ given by
$$S(x)=\frac{1-x-\sqrt{1-6x+x^2}}{2x}.$$
We can express this as the following continued fraction.
$$S(x)=\cfrac{1}{1-x-\cfrac{x}{1-x-\cfrac{x}{1-x-\cdots}}}.$$
This can be seen by solving the equation
$$u=\frac{1}{1-x-xu}$$ and taking the solution with $u(0)=0$.

Continued fractions of the form
$$\cfrac{1}{1-ax - \cfrac{bx}{1- cx - \cfrac{dx}{1-e x- \cdots}}}$$ are called Thron or $T$-continued fractions. We then have the following result \cite{Josuat, Oste, Sokal}.
$$\cfrac{1}{1-\alpha_0 x-\cfrac{\beta_1 x}{1-\alpha _1 x-\cfrac{\beta_2 x}{1-\alpha_2 x-\cdots}}}=\sum_{n=0}^{\infty} T_n(\mathbf{\alpha}, \mathbf{\beta})x^n,$$
where $T_n(\mathbf{\alpha}, \mathbf{\beta})$ is the generating function for Schr\"oder paths of length $2n$ in which each rise has weight $1$, each horizontal step at height $i$ has weight $\alpha_i$, and each fall from height $i$ has weight $\beta_i$.

\section{Results}

We are now in a position to describe certain families of Riordan arrays which are defined by  Jacobi and Thron continued fractions.

\begin{proposition} The generating function
$$G(x,y)=\cfrac{1}{1-(y+a)x-\cfrac{(y+b)x^2}{1-c x -\cfrac{dx^2}{1-cx - \cdots}}}$$ is the generating function of the Riordan array $(g(x), f(x))$ where
$$g(x)=\frac{2d}{b \sqrt{1-2cx+(c^2-4d)x^2}+(bc-2ad)-b+2d},$$ and
$$f(x)=\frac{(1+(a-b)x)\sqrt{1-2cx+(c^2-4d)x^2}+(a(c-2d)+b(c-2))x^2-(a+b+c-2d)x+1}{2((a^2d-abc+b^2)x^2+a(b-2d)x+bc-b+d)}.$$ \end{proposition}.
\begin{proof}
We let $u=u(x)$ be defined by
$$u=\frac{1}{1- cx - dx^2 u}. $$
Then
$$G(x,y)=\frac{1}{1-(y+a)x-(y+b)x^2 u(x)}.$$
Now let
$\mathfrak{g}(x)=G(x,0)$ and $\mathfrak{f}(x)=1-\frac{G(x,0)}{G(x,1)}$.
We find that in fact we have
$$g(x)=\mathfrak{g}(x),\quad f(x)=\mathfrak{f}(x),$$ and evaluation then shows that
$$G(x,y)=\frac{g(x)}{1-y f(x)}.$$

\end{proof}
We now turn to the inverse of the Riordan array of the above proposition. Note that the Jacobi continued fraction of the proposition can be written as
$$\mathcal{J}(y+a,c,c, \ldots; y+b,d,d,\ldots).$$
\begin{proposition} Let $(g(x), f(x))$ be the Riordan array with generating function given by
$$G(x,y)=\mathcal{J}(y+a,c,c,c, \ldots; y+b,d,d,d,\ldots).$$
Then the generating function of the inverse Riordan array $(g(x), f(x))^{-1}$ is given by the Jacobi continued fraction $$\mathcal{J}(y-a,c-a-2,c-a-2,\ldots; a-b-y,1+a-b-c+d, 1+a-b-c+d,\ldots).$$
\end{proposition}
\begin{proof} The steps of the proof are essentially the same as in the previous proposition. Thus let
$$\bar{G}(x,y)=\mathcal{J}(y-a,c-a-2,\ldots; a-b-y,1+a-b-c+d, \ldots),$$ and form
$$\bar{\mathfrak{g}}(x)=\bar{G}(x,0),\quad \bar{\mathfrak{f}}(x)=1-\frac{\bar{G}(x,0)}{\bar{G}(x,1)}.$$
Then we find that
$$\bar{\mathfrak{g}}(x)=\frac{1}{g(\bar{f}(x))}, \quad \bar{\mathfrak{f}}(x)=\bar{f}(x),$$  and that
$$\bar{G}(x,y)=\frac{\frac{1}{g(\bar{f}(x))}}{1-y \bar{f}(x)}.$$
\end{proof}
\begin{example} We consider the generating function
$$G(x,y)=\cfrac{1}{1-(y+2)x-\cfrac{(y+3)x}{1-x-\cfrac{4x^2}{1-x-\cdots}}}.$$
Then $G(x,y)$ is the generating function of the Riordan array
$$\left(\frac{5-13x-3\sqrt{1-2x-15x^2}}{2(1-7x+19x^2)}, \frac{1+2x-17x^2-(1+x)\sqrt{1-2x-15x^2}}{2(1-7x+19x^2)}\right).$$
This Riordan array begins
$$\left(
\begin{array}{ccccc}
 1 & 0 & 0 & 0 & 0 \\
 2 & 1 & 0 & 0 & 0 \\
 7 & 5 & 1 & 0 & 0 \\
 23 & 23 & 8 & 1 & 0 \\
 88 & 101 & 48 & 11 & 1 \\
\end{array}
\right).$$
The inverse array is then given by
$$\left(\frac{6}{7+15x-\sqrt{1+6x-3x^2}}, \frac{(1+3x)\sqrt{1+6x-3x^2}+7x^2+2x-1}{2(4+17x+19x^2)}\right),$$ whose generating function $\bar{G}(x,y)$ is given by the Jacobi continued fraction
$$\bar{G}(x,y)=\cfrac{1}{1-(y-2)x+\cfrac{(y+1)x^2}{1+3x-\cfrac{3x^2}{1+3x-\cdots}}}.$$
We have, for instance,
$$\left(
\begin{array}{ccccc}
 1 & 0 & 0 & 0 & 0 \\
 2 & 1 & 0 & 0 & 0 \\
 7 & 5 & 1 & 0 & 0 \\
 23 & 23 & 8 & 1 & 0 \\
 88 & 101 & 48 & 11 & 1 \\
\end{array}
\right)\left(\begin{array}{c}
1\\
2\\
4\\
8\\
16\\ \end{array}\right)=
\left(\begin{array}{c}
1\\
4\\
21\\
109\\
586\\ \end{array}\right).$$
We can interpret this to mean that the sequence $1,4,21,109,\ldots$ counts Motzkin paths whose horizontal steps at level $0$ have $2+2=4$ colors, whose rise steps at level $0$ have $2+3=5$ colors, and whose horizontal steps thereafter have $1$ color, and whose rise steps after level $0$ have $4$ colors.
\end{example}

Using the same methods of proof as above, we can obtain the following result.
\begin{proposition} The generating function of the Riordan array
$$\left(\frac{1+cx}{1+ax+bx^2}, \frac{x}{1+ax+bx^2}\right)^{-1}$$ is given by the Jacobi continued fraction
$$\cfrac{1}{1-(y+a-c)x-\cfrac{(cy+b)x^2}{1-ax-\cfrac{bx^2}{1-ax-\cdots}}}.$$
\end{proposition}
We can in fact strengthen our first result as follows.
\begin{proposition} Let
$$G(x,y)=\cfrac{1}{1-(a+by)x-\cfrac{(c+dy)x^2}{1-ux-\frac{vx^2}{1-ux-\cdots}}}.$$
Then $G(x,y)$ is the generating function of the Riordan array $(g(x), f(x))$ where
$$g(x)=\frac{2v}{2v-c+(cu-2av)x+c \sqrt{1-2ux+(u^2-4v)x^2}},$$ and
$$f(x)=\frac{d-(du-2bv)x-d \sqrt{1-2ux+(u^2-4v)x^2}}{2v-c+(cu-2av)x+c \sqrt{1-2ux+(u^2-4v)x^2}}.$$
\end{proposition}
We now turn to Thron continued fractions. We have the following result.
\begin{proposition}
We let
$$G(x,y)=\cfrac{1}{1-(ay+b)x-\cfrac{(cy+d)x}{1-ux-\cfrac{vx}{1-ux-\cdots}}}.$$
Then $G(x,y)$ is the generating function of the Riordan array $(g(x),f(x))$ where
$$g(x)=\frac{2v}{d(2v-d+(du-2bv)x+\sqrt{1-2(u+2v)x+u^2 x^2}}$$ and
$$f(x)=\frac{c-(cu-2av)x-c \sqrt{1-2(u+2v)x+u^2x^2}}{2v-d+(du-2bv)x+d\sqrt{1-2(u+2v)x+u^2x^2}}.$$
\end{proposition}
\begin{proof}
We solve the equation
$$z=\frac{1}{1-ux-vx z}$$ to form
$$G(x,y)=\frac{1}{1-(ay+b)x-(cy+d)x z}.$$
We again set $\mathfrak{g}(x)=G(x,0)$ and $\mathfrak{f}(x)=1-\frac{G(x,0)}{G(x,1)}$.
We then confirm that
$$\frac{\mathfrak{g}(x)}{1-y \mathfrak{f}(x)}=G(x,y).$$
We then set $g(x)=\mathfrak{g}(x)$, $f(x)=\mathfrak{f}(x)$.
\end{proof}
In fact, the Thron continued fraction above has an equivalent Jacobi continued fraction expression. This is the content of the next result.
\begin{proposition} We let $G(x,y)$ be the following Thron continued fraction.
$$G(x,y)=\cfrac{1}{1-(ay+b)x-\cfrac{(cy+d)x}{1-ux-\cfrac{vx}{1-ux-\cdots}}}.$$
Then $G(x,y)$ can be expressed as the following Jacobi continued fraction.
$$G(x,y)=\cfrac{1}{1-(b+d+(a+c)y)x-\cfrac{(u+v)(d+cy)x^2}{1-(u+2v)x-\cfrac{v(u+v)x^2}{1-(u+2v)x-\cdots}}}.$$
\end{proposition}
\begin{proof} We can show that both continued fractions give the generating function of the same Riordan array.
\end{proof}
In particular, the Thron continued fraction
$$\frac{1}{1-(ay+b)x-\cfrac{((1-a)y+d)x}{1-ux-\cfrac{vx}{1-ux-\cdots}}}$$ is equivalent to the Jacobi continued fraction
$$\frac{1}{1-(b+c+y)x-\cfrac{(u+v)(d+y(1-a))x^2}{1-(u+2v)x-\cfrac{v(u+v)x^2}{1-(u+2v)x-\cdots}}}.$$
The Riordan array corresponding to these two expressions will by construction have $1$s on the diagonal.

The foregoing result can be interpreted in terms of the enumeration of colored Motzkin and Schr\"oder paths. For instance, taking $a=2, b=3, c=1-a=-1, d=4, u=1$ and $v=5$, we can say that Schr\"oder paths of length $2n$ with horizontal steps of $1$ color and rise steps of $5$ colors, except at level $0$, where the horizontal steps have $2y+3$ colors and the rise steps have $-y+4$ colors, are equinumerous with Motzkin paths of length $n$ whose horizontal steps have $1+2\cdot5=11$ colors and whose rise steps have $5(1+5)=30$ colors, except at level $0$, where the horizontal steps have $3-1+y=2+y$ colors, and the rise steps have $(1+5)(4+y(-1))=24-6y$ colors.

\section{Examples}
It is possible that a Riordan array has a generating function that can be expressed as a Stieltjes, a Jacobi and a Thron continued fraction. We note that should the generating function have a Stieltjes continued fraction expression, then it will automatically have a Jacobi continued fraction expression.
\begin{example} The following three continued fractions are equal to the generating function $G(x,y)$ of the Riordan array
$$\left(1, \frac{x(1-x)}{1+x}\right)^{-1}=\left(1,\frac{1-x-\sqrt{1-6x+x^2}}{2}\right).$$
\begin{align*}
G(x,y)&=\cfrac{1}{1-\cfrac{xy}{1-\cfrac{2x}{1-\cfrac{x}{1-\cfrac{2x}{1-\cdots}}}}}\\
&=\cfrac{1}{1-yx-\cfrac{2yx^2}{1-3x-\cfrac{2x^2}{1-3x-\cfrac{2x^2}{1-3x-\cdots}}}}\\
&=\cfrac{1}{1-\cfrac{yx}{1-x-\cfrac{x}{1-x-\cfrac{x}{1-x-\cdots}}}}.\end{align*}

This array begins
$$\left(
\begin{array}{cccccc}
 1 & 0 & 0 & 0 & 0 & 0 \\
 0 & 1 & 0 & 0 & 0 & 0 \\
 0 & 2 & 1 & 0 & 0 & 0 \\
 0 & 6 & 4 & 1 & 0 & 0 \\
 0 & 22 & 16 & 6 & 1 & 0 \\
 0 & 90 & 68 & 30 & 8 & 1 \\
\end{array}
\right).$$

The production matrix of this array begins
$$\left(
\begin{array}{cccccc}
 0 & 1 & 0 & 0 & 0 & 0 \\
 0 & 2 & 1 & 0 & 0 & 0 \\
 0 & 2 & 2 & 1 & 0 & 0 \\
 0 & 2 & 2 & 2 & 1 & 0 \\
 0 & 2 & 2 & 2 & 2 & 1 \\
 0 & 2 & 2 & 2 & 2 & 2 \\
\end{array}
\right).$$
The generating function of its inverse $\left(1, \frac{x(1-x)}{1+x}\right)$ can be expressed as the finite Jacobi fraction
$$\frac{1+x}{1+x(1-y)x+x^2y}=\cfrac{1}{1-yx+\cfrac{2yx^2}{1+x}}.$$
We now consider the Riordan array whose production matrix begins
$$\left(
\begin{array}{cccccc}
 1 & 1 & 0 & 0 & 0 & 0 \\
 1 & 2 & 1 & 0 & 0 & 0 \\
 1 & 2 & 2 & 1 & 0 & 0 \\
 1 & 2 & 2 & 2 & 1 & 0 \\
 1 & 2 & 2 & 2 & 2 & 1 \\
 1 & 2 & 2 & 2 & 2 & 2 \\
\end{array}
\right).$$ This is the Riordan array $\left(1+\frac{1-x-\sqrt{1-6x+x^2}}{2},\frac{1-x-\sqrt{1-6x+x^2}}{2}\right)$ that begins
$$\left(
\begin{array}{cccccc}
 1 & 0 & 0 & 0 & 0 & 0 \\
 1 & 1 & 0 & 0 & 0 & 0 \\
 2 & 3 & 1 & 0 & 0 & 0 \\
 6 & 10 & 5 & 1 & 0 & 0 \\
 22 & 38 & 22 & 7 & 1 & 0 \\
 90 & 158 & 98 & 38 & 9 & 1 \\
\end{array}
\right).$$ Its inverse $\left(\frac{1}{1+x}, \frac{x(1-x)}{1+x}\right)$ is the signed version of the Delannoy number triangle that begins
$$\left(
\begin{array}{cccccc}
 1 & 0 & 0 & 0 & 0 & 0 \\
 -1 & 1 & 0 & 0 & 0 & 0 \\
 1 & -3 & 1 & 0 & 0 & 0 \\
 -1 & 5 & -5 & 1 & 0 & 0 \\
 1 & -7 & 13 & -7 & 1 & 0 \\
 -1 & 9 & -25 & 25 & -9 & 1 \\
\end{array}
\right).$$ The generating function of the Riordan array $\left(1+\frac{1-x-\sqrt{1-6x+x^2}}{2},\frac{1-x-\sqrt{1-6x+x^2}}{2}\right)$ can be expressed as the following Jacobi continued fraction.
$$\cfrac{1}{1-(y+1)x-\cfrac{(y+1)x^2}{1-3x-\cfrac{2x^2}{1-3x-\cfrac{2x^2}{1-3x-\cdots}}}}.$$
Its inverse, the signed Delannoy number triangle, has its generating function given by
$$ \frac{1}{1-(y-1)x+yx^2} = \cfrac{1}{1+x-\cfrac{yx}{1+\cfrac{x}{1-x}}},$$
that is, a Jacobi fraction and a Thron fraction (of finite type).
We next look at the Riordan array whose production matrix begins
$$\left(
\begin{array}{cccccc}
 2 & 1 & 0 & 0 & 0 & 0 \\
 2 & 2 & 1 & 0 & 0 & 0 \\
 2 & 2 & 2 & 1 & 0 & 0 \\
 2 & 2 & 2 & 2 & 1 & 0 \\
 2 & 2 & 2 & 2 & 2 & 1 \\
 2 & 2 & 2 & 2 & 2 & 2 \\
\end{array}
\right).$$
This is the Riordan array
$$\left(\frac{1-x-\sqrt{1-6x+x^2}}{2x}, \frac{1-x-\sqrt{1-6x+x^2}}{2}\right)=\left(\frac{1-x}{1+x}, \frac{x(1-x)}{1+x}\right)^{-1}.$$
We can express the generating function of this array as a Jacobi fraction
$$\cfrac{1}{1-(y+2)x-\cfrac{2x^2}{1-3x-\cfrac{2x^2}{1-3x-\cdots}}},$$ or as a Thron continued fraction
$$\cfrac{1}{1-(y+1)x-\cfrac{x}{1-x-\cfrac{x}{1-x-\cdots}}}.$$
The generating function of the inverse array can be expressed in terms of (finite) Jacobi and Thron continued fractions:
$$\cfrac{1}{1+(y-2)x+\cfrac{2x^2}{1-x}}=\cfrac{1}{1-yx+\cfrac{2x}{1+x-2x}}.$$
\end{example}
\begin{example}
We consider the Riordan array $\left(\frac{3x-\sqrt{1-6x+x^2}}{2}, x\frac{3x-\sqrt{1-6x+x^2}}{2}\right)$ whose generating function can be expressed as the Jacobi continued fraction
$$\cfrac{1}{1-(y+1)x-\cfrac{x^2}{1-3x-\cfrac{2x^2}{1-3x-\cdots}}},$$ or by the Thron continued fraction
$$\cfrac{1}{1-xy-\cfrac{x}{1-\cfrac{x}{1-x-\cfrac{x}{1-x-\cdots}}}}.$$
It is thus associated to Schr\"oder paths with no horizontal step at level $1$. This array begins
$$\left(
\begin{array}{cccccc}
 1 & 0 & 0 & 0 & 0 & 0 \\
 1 & 1 & 0 & 0 & 0 & 0 \\
 2 & 2 & 1 & 0 & 0 & 0 \\
 6 & 5 & 3 & 1 & 0 & 0 \\
 22 & 16 & 9 & 4 & 1 & 0 \\
 90 & 60 & 31 & 14 & 5 & 1 \\
\end{array}
\right).$$

Its inverse is the Riordan array $\left(\frac{1}{1+xc(x)}, \frac{x}{1+xc(x)}\right)$ whose generating function may be expressed by the Jacobi continued fraction
$$\cfrac{1}{1-(y-1)x+\cfrac{x^2}{1-2x-\cfrac{x^2}{1-2x-\cfrac{x^2}{1-2x-\cdots}}}}.$$
\end{example}

\section{Riordan involutions}
An \emph{involution} in the Riordan group is an element of order $2$, defined by an element $(g(x),f(x))$ for which we have
$$(g(x), f(x))^2 = (g(x), f(x))\cdot (g(x), f(x))=(g(x) g(f(x)),f(f(x)))=(1,x).$$
Thus we require $\bar{f}(x)=f(x)$, that is, $f$ is its own compositional inverse. Examples of such power series $f(x)$ are given by $f(x)=-\frac{x}{1-x}$ and $f(x)=-\frac{x}{1+x}$. Clearly then $\left(1, -\frac{x}{1-x}\right)$ and $\left(1, -\frac{x}{1+x}\right)$ are involutions in the Riordan group. A family of less trivial ones is given by the following result, based on results on moments and involutions \cite{Inv} associated with the Riordan group.
\begin{proposition} The generating function
$$G(x,y)=\cfrac{1}{1-(-y+2a-2)x+\cfrac{(y-a+1)x^2}{1-ax-\cfrac{bx^2}{1-ax-\cdots}}}$$ is the generating function of the Riordan involution $(g(x), f(x))$ where
$$g(x)=\frac{2b}{1-a+2b+(a-1)(a-4b)x+(a-1)\sqrt{1-2ax+(a^2-4b)x^2}},$$ and
$$f(x)=\frac{\sqrt{1-2ax+(a^2-4b)x^2}+(a-2b)x-1}{1-a+2b+(a-1)(a-4b)x+(a-1)\sqrt{1-2ax+(a^2-4b)x^2}}.$$
In addition, the Riordan array $(g(x), f(x))$ is the coefficient array of the parameterized moments of the family of orthogonal polynomials whose coefficient array is given by the Riordan array
$$\left(\frac{1+(2-a+y)x+(-a+b+1+y)x^2}{1+ax+bx^2}, \frac{x}{1+ax+bx^2}\right).$$
\end{proposition}
\begin{example} When $a=b=1$, we obtain the Riordan involution
$$\left(1, \frac{\sqrt{1-2x-3x^2}-x-1}{2}\right)$$ which begins
$$\left(
\begin{array}{cccccc}
 1 & 0 & 0 & 0 & 0 & 0 \\
 0 & -1 & 0 & 0 & 0 & 0 \\
 0 & -1 & 1 & 0 & 0 & 0 \\
 0 & -1 & 2 & -1 & 0 & 0 \\
 0 & -2 & 3 & -3 & 1 & 0 \\
 0 & -4 & 6 & -6 & 4 & -1 \\
\end{array}
\right).$$
The sequence of polynomials for which this matrix is the coefficient array thus begins
$$1, -y, y^2 - y, - y^3 + 2y^2 - y, y^4 - 3y^3 + 3y^2 - 2y,\ldots.$$
This sequence forms the first column in the inverse of the Riordan array
$$\left(\frac{1+(y+1)x+(y+1)x^2}{1+x+x^2}, \frac{x}{1+x+x^2}\right).$$
We note that the production matrix of $\left(\frac{1+(y+1)x+(y+1)x^2}{1+x+x^2}, \frac{x}{1+x+x^2}\right)^{-1}$ is the tri-diagonal matrix that begins
$$\left(
\begin{array}{ccccc}
 -y & 1 & 0 & 0 & 0 \\
 -y & 1 & 1 & 0 & 0 \\
 0 & 1 & 1 & 1 & 0 \\
 0 & 0 & 1 & 1 & 1 \\
 0 & 0 & 0 & 1 & 1 \\
\end{array}
\right).$$
When $y=-1$ (and $a=b=1$) the moment sequence is given by the Motzkin numbers. When $y=-2$ (and $a=b=1$) we obtain the sequence $1, 2, 6, 18, 56, 176, 558,\ldots$ with generating function $$\frac{1}{\sqrt{1-2x-3x^2}-x}
=\cfrac{1}{1-2x-\cfrac{2x^2}{1-x-\cfrac{x}{1-x-\cdots}}}.$$

This sequence therefore counts Motzkin paths of length $n$ where the horizontal and rise steps at level $0$ both have two colors.
\end{example}
\begin{example} We consider the case $a=b=2$. We find that
$$(g(x), f(x))=\left(\frac{3(1-2x)-\sqrt{1-4x-4x^2}}{2(1-4x+5x^2)},\frac{(1-x)\sqrt{1-4x-4x^2}+4x^2+x-1}{2(1-4x+5x^2)}\right).$$ The involution $(g(x), f(x))$ begins
$$\left(
\begin{array}{cccccc}
 1 & 0 & 0 & 0 & 0 & 0 \\
 2 & -1 & 0 & 0 & 0 & 0 \\
 5 & -5 & 1 & 0 & 0 & 0 \\
 14 & -20 & 8 & -1 & 0 & 0 \\
 43 & -76 & 44 & -11 & 1 & 0 \\
 142 & -287 & 210 & -77 & 14 & -1 \\
\end{array}
\right).$$
We have
$$G(x,y)=\cfrac{1}{1-(2-y)x-\cfrac{(1-y)x^2}{1-2x-\cfrac{2x^2}{1-2x-\cdots}}},$$ and thus we have
$$g(x)=G(x,0)=\cfrac{1}{1-2x-\cfrac{x^2}{1-2x-\cfrac{2x^2}{1-2x-\cdots}}}.$$
We have
$$\left(
\begin{array}{cccccc}
 1 & 0 & 0 & 0 & 0 & 0 \\
 2 & -1 & 0 & 0 & 0 & 0 \\
 5 & -5 & 1 & 0 & 0 & 0 \\
 14 & -20 & 8 & -1 & 0 & 0 \\
 43 & -76 & 44 & -11 & 1 & 0 \\
 142 & -287 & 210 & -77 & 14 & -1 \\
\end{array}
\right)\left(\begin{array}{c}
1\\
-1\\
1\\
-1\\
1\\
-1\\ \end{array}\right)=
\left(\begin{array}{c}
1\\
3\\
11\\
43\\
175\\
731\\ \end{array}\right).$$
Thus the sequence $1,3,11,43,175,\ldots$ with generating function
$$\frac{3(1-2x)-\sqrt{1-4x-4x^2}}{1-7x+14x^2+(1-x)\sqrt{1-4x-4x^2}}$$ or equivalently
$$\cfrac{1}{1-3x-\cfrac{2x^2}{1-2x-\cfrac{2x^2}{1-2x-\cdots}}}$$
counts Motzkin paths whose horizontal and rise steps after level $0$ both have $2$ colors, and whose horizontal steps at level $0$ have $-(-1)+2=3$ colors and whose rise steps at level $0$ have $1-(-1)=2$ colors. This is sequence \seqnum{A151090} \cite{Chang}.
\end{example}

\section{Laurent biorthogonal polynomials}
To better understand Thron continued fractions, we give an example of a Riordan array which defines a family of Laurent biorthogonal polynomials, whose moment sequence has its generating function given by a Thron continued fraction \cite{LBP, Sawa1, Sawa2}.
\begin{example}
We consider the family of Laurent biorthogonal polynomials defined by the constant coefficient recurrence 
$$P_n(x)=(x+1) P_{n-1}(x)-3xP_{n-2}(x),$$ with $P_0(x)=1$, $P_1(x)=x-1$. These polynomials begin
$$1, x - 1, x^2 - 3x - 1, x^3 - 5x^2 - x - 1, x^4 - 7x^3 + 3x^2 + x - 1,\ldots.$$ The coefficient array of these polynomials is the Riordan array $\left(\frac{1-2x}{1-x}, \frac{x(1-3x)}{1-x}\right)$ which begins
$$\left(
\begin{array}{cccccc}
 1 & 0 & 0 & 0 & 0 & 0 \\
 -1 & 1 & 0 & 0 & 0 & 0 \\
 -1 & -3 & 1 & 0 & 0 & 0 \\
 -1 & -1 & -5 & 1 & 0 & 0 \\
 -1 & 1 & 3 & -7 & 1 & 0 \\
 -1 & 3 & 7 & 11 & -9 & 1 \\
\end{array}
\right).$$ The inverse array is given by $\left(\frac{4}{3+x+\sqrt{1-10x+x^2}}, \frac{1-x-\sqrt{1-10x+x^2}}{4x}\right)$ which begins
$$\left(
\begin{array}{cccccc}
 1 & 0 & 0 & 0 & 0 & 0 \\
 1 & 1 & 0 & 0 & 0 & 0 \\
 4 & 3 & 1 & 0 & 0 & 0 \\
 22 & 16 & 5 & 1 & 0 & 0 \\
 142 & 102 & 32 & 7 & 1 & 0 \\
 1006 & 718 & 226 & 52 & 9 & 1 \\
\end{array}
\right).$$ 
The first column sequence of this array, which begins 
$$1, 1, 4, 22, 142, 1006, 7570, 59410, 480910,\ldots,$$ is thus the moment sequence of the Laurent biorthogonal polynomials $P_n(x)$. We can express the generating function $g(x)=\frac{4}{3+x+\sqrt{1-10x+x^2}}$ of this sequence in the following Thron continued fraction form.
$$\cfrac{1}{1-\cfrac{x}{1-x-\cfrac{2x}{1-x-\cfrac{2x}{1-x-\cdots}}}}.$$ 
It is of interest to note that $g(x)$ can also be expressed as a Stieltjes continued fraction, 
$$g(x)=\cfrac{1}{1-\cfrac{x}{1-\cfrac{3x}{1-\cfrac{2x}{1-\cfrac{3x}{1-\cdots}}}}},$$ from which it follows that it also has the following Jacobi continued fraction expression.
$$g(x)=\cfrac{1}{1-x-\cfrac{3x^2}{1-5x-\cfrac{6x^2}{1-5x-\cfrac{6x^2}{1-5x-\cdots}}}}.$$ 
The generating function $G(x,y)$ of the moment matrix $\left(\frac{4}{3+x+\sqrt{1-10x+x^2}}, \frac{1-x-\sqrt{1-10x+x^2}}{4x}\right)$ will then have the following Jacobi continued fraction expression.
$$G(x,y)=\cfrac{1}{1-(y+1)x-\cfrac{(y+3)x^2}{1-5x-\cfrac{6x^2}{1-5x-\cdots}}}.$$
We can also represent this as a Thron continued fraction. We have the following.
$$G(x,y)=\cfrac{1}{1-\frac{2}{3}yx-\cfrac{\frac{y+3}{3}x}{1-x-\cfrac{2x}{1-x-\cfrac{2x}{1-x-\cdots}}}}.$$
Both of these expressions evaluate to 
$$G(x,y)=\frac{12}{9-y+(3-7y)x+(y+3)\sqrt{1-10x+x^2}}.$$ 
\end{example}
Just as there is a close relationship between Jacobi continued fractions, orthogonal polynomials and Hankel determinants, there is a close relationship between Thron continued fractions, Laurent biorthogonal polynomials, and Toeplitz determinants \cite{LBP}.

\section{From ordinary to exponential Riordan arrays}
An exponential Riordan array is defined by a pair of formal power series
$$g(x)=g_0 + g_1 \frac{x}{1!}+ g_2 \frac{x^2}{2!}+ \cdots,$$
$$f(x)=f_1 \frac{x}{1!} + f_2 \frac{x^2}{2!} + \cdots,$$ where again we specify that
$g_0 \ne 0$, $f_0=0$ and $f_1 \ne 0$. In this case, the matrix associated to the pair $(g(x), f(x))$ will have its $(n,k)$-th element given by
$$\frac{n!}{k!} [x^n] g(x)f(x)^k.$$
When $g(x)$ and $f(x)$ are exponential generating functions we write $[g(x), f(x)]$ to denote this exponential Riordan array.
\begin{example} The exponential Riordan array $\left[e^x, x\right]$ is the binomial matrix, since we have
$$\frac{n!}{k!} [x^n] e^x x^k = \frac{n!}{k!} [x^{n-k}] e^x =\frac{n!}{k!} \frac{1}{(n-k)!}=\binom{n}{k}.$$
\end{example}
We now consider the ordinary Riordan array
$$\left(\frac{1-x-\sqrt{1-6x+x^2}}{2x}, \frac{1-x-\sqrt{1-6x+x^2}}{2}\right)=\left(\frac{1-x}{1+x}, \frac{x(1-x)}{1+x}\right)^{-1}.$$
The generating function of this array can be expressed as the Jacobi continued fraction
$$\cfrac{1}{1-(y+2)x-\cfrac{2x^2}{1-3x-\cfrac{2x^2}{1-3x-\cfrac{2x^2}{1-3x-\cdots}}}}.$$
We now consider the number array whose (ordinary) generating function is given by the Jacobi continued fraction
$$\cfrac{1}{1-(y+2)x-\cfrac{2\cdot 1x^2}{1-(y+5)x-\cfrac{2\cdot4x^2}{1-(y+8)x-\cfrac{2\cdot 9x^2}{1-(y+11)x-\cdots}}}}.$$
Here, we have replaced the coefficients of $x$, given by $y+2,3,3,\ldots$, by their partial sums, and we have introduced factors $1,4,9,\ldots$ to the coefficients of $x^2$. The resulting number triangle begins
$$\left(
\begin{array}{ccccccc}
 1 & 0 & 0 & 0 & 0 & 0 & 0 \\
 2 & 1 & 0 & 0 & 0 & 0 & 0 \\
 6 & 4 & 1 & 0 & 0 & 0 & 0 \\
 26 & 18 & 6 & 1 & 0 & 0 & 0 \\
 150 & 104 & 36 & 8 & 1 & 0 & 0 \\
 1082 & 750 & 260 & 60 & 10 & 1 & 0 \\
 9366 & 6492 & 2250 & 520 & 90 & 12 & 1 \\
\end{array}
\right).$$
This is the exponential Riordan array $\left[\frac{e^x}{2-e^x}, x\right]$. Its first column, with exponential generating function $\frac{e^x}{2-e^x}$, is sequence \seqnum{A000629} (the binomial transform of the Fubini numbers \seqnum{A000670}), which counts the number of necklaces of partitions of $n+1$ labeled beads. The row sums of this matrix are the binomial transform of this sequence, namely \seqnum{A007047}, which gives the number of chains in the power set of an $n$-set.

If instead of the multipliers $1,4,9,\ldots$ we use the multipliers $1,3,6,10,\ldots$ we obtain the number triangle that begins
$$\left(
\begin{array}{cccccc}
 1 & 0 & 0 & 0 & 0 & 0 \\
 2 & 1 & 0 & 0 & 0 & 0 \\
 6 & 4 & 1 & 0 & 0 & 0 \\
 26 & 18 & 6 & 1 & 0 & 0 \\
 146 & 104 & 36 & 8 & 1 & 0 \\
 994 & 730 & 260 & 60 & 10 & 1 \\
\end{array}
\right).$$
This is the Riordan array $$\left[\frac{10(3\sqrt{5}+7)e^{-x}e^{3\sqrt{5}x}}{(2e^{2\sqrt{5}x}-(3\sqrt{5}+7)e^{\sqrt{5}x})^2}, x\right]$$

whose row sums (essentially \seqnum{A230008})
$$1, 3, 11, 51, 295, 2055,\ldots$$

are the (exponential) revert transform of $n![x^n]\frac{1}{1+3x+x^2}$.

Finally, if we change the multipliers to $1,2,3,\ldots$ we find that the generating function
$$G(x,y)=\cfrac{1}{1-(y+2)x-\cfrac{2\cdot 1x^2}{1-(y+5)x-\cfrac{2\cdot2x^2}{1-(y+8)x-\cfrac{2\cdot 3x^2}{1-(y+11)x-\cdots}}}}$$ is the (ordinary) generating function of the exponential Riordan array
$$\left[  e^{\frac{2e^{3x}}{9}+\frac{4x}{3}-\frac{2}{9}}, x\right].$$ This array begins
$$\left(
\begin{array}{cccccc}
 1 & 0 & 0 & 0 & 0 & 0 \\
 2 & 1 & 0 & 0 & 0 & 0 \\
 6 & 4 & 1 & 0 & 0 & 0 \\
 26 & 18 & 6 & 1 & 0 & 0 \\
 142 & 104 & 36 & 8 & 1 & 0 \\
 906 & 710 & 260 & 60 & 10 & 1 \\
\end{array}
\right).$$
We formalize this result in the following proposition.
\begin{proposition}
The exponential Riordan array
$$\left[e^{\frac{2e^{3x}}{9}+x\left(y+\frac{4}{3}\right)-\frac{2}{9}}, \frac{1}{3}\left(3^{3x}-1\right)\right]$$ is the moment array of the family of orthogonal polynomials whose moments have the generating function $G(x,y)$ above.
\end{proposition}
\begin{proof}
The production matrix of the exponential Riordan array $[g(x), f(x)]$ has its generating function given by
$$e^{xz}(Z(x)+zA(x)),$$ where
$$Z(x)=\frac{g'(\bar{f})}{g(\bar{f})}, \quad A(x)=f'(\bar{f}).$$
With $g(x)=e^{\frac{2e^{3x}}{9}+x\left(y+\frac{4}{3}\right)-\frac{2}{9}}$ and $f(x)=\frac{1}{3}\left(3^{3x}-1\right)$, we obtain the expression
$$e^{xz}(y+2+2x+z(1+3x))$$ for the generating function of the production matrix of the exponential Riordan array of the proposition \cite{Prod}. This expands to give the tri-diagonal matrix that begins
$$\left(
\begin{array}{cccccc}
 y+2 & 1 & 0 & 0 & 0 & 0 \\
 2 & y+5 & 1 & 0 & 0 & 0 \\
 0 & 4 & y+8 & 1 & 0 & 0 \\
 0 & 0 & 6 & y+11 & 1 & 0 \\
 0 & 0 & 0 & 8 & y+14 & 1 \\
 0 & 0 & 0 & 0 & 10 & y+17 \\
\end{array}
\right).$$
The result follows from this.
\end{proof}
\begin{corollary} We have that
$$G(x,y)=\cfrac{1}{1-(y+2)x-\cfrac{2\cdot 1x^2}{1-(y+5)x-\cfrac{2\cdot2x^2}{1-(y+8)x-\cfrac{2\cdot 3x^2}{1-(y+11)x-\cdots}}}}$$ is the (ordinary) generating function of the exponential Riordan array
$$\left[  e^{\frac{2e^{3x}}{9}+\frac{4x}{3}-\frac{2}{9}}, x\right].$$
\end{corollary}
\begin{proof} The generating function of the exponential Riordan array $[g(x), f(x)]$ is given by
$g(x)e^{yg(x)}$. By the proposition, $G(x,y)$ is the (ordinary) generating function moment sequence whose exponential generating function is given by $e^{\frac{2e^{3x}}{9}+x\left(y+\frac{4}{3}\right)-\frac{2}{9}}$.
Writing this as $e^{\frac{2e^{3x}}{9}+x\left(\frac{4}{3}\right)-\frac{2}{9}}e^{yx}$ we see that it has the required form.
\end{proof}
The same method of proof assures us that the results stated above for the multipliers $1,3,6,\ldots$ and $1,4,9,\ldots$ also hold.

\section{The Narayana and related number triangles and continued fractions}
By way of contrast with the foregoing sections which treated Riordan arrays, in this section we look at some combinatorially important number triangles which are not Riordan arrays. The common feature of these triangles is that their generating functions are given by (parameterized) continued fractions. Again, we meet Stieltjes, Jacobi and Thron continued fractions. We let $\mathbf{B}$ denote the binomial matrix, which represents the Riordan array $\left(\frac{1}{1-x}, \frac{x}{1-x}\right)$.

The Narayana numbers $N(n,k)=\frac{1}{k+1}\binom{n}{k}\binom{n+1}{k}$ form the triangle $\mathbf{N}$ \seqnum{A001263} that begins
$$\left(
\begin{array}{cccccc}
 1 & 0 & 0 & 0 & 0 & 0 \\
 1 & 1 & 0 & 0 & 0 & 0 \\
 1 & 3 & 1 & 0 & 0 & 0 \\
 1 & 6 & 6 & 1 & 0 & 0 \\
 1 & 10 & 20 & 10 & 1 & 0 \\
 1 & 15 & 50 & 50 & 15 & 1 \\
\end{array}
\right).$$
Its generating function is given by the Jacobi continued fraction \cite{Nara}
$$N(x,y)=\cfrac{1}{1-(y+1)x-\cfrac{yx^2}{1-(y+1)x-\cfrac{yx^2}{1-(y+1)x-\cdots}}}.$$ We can also express this generating function as a Thron continued fraction.
\begin{proposition} We have
$$N(x,y)=\cfrac{1}{1-x-\cfrac{yx}{1+(y-1)x-\cfrac{yx}{1+(y-1)x-\cfrac{yx}{1+(y+1)x-\cdots}}}}.$$
\end{proposition}
\begin{proof} We calculate $z$ where
$$z=\frac{1}{1+(y-1)x-yx z},$$ taking the value with $z(0)=1$.
We then simplify $\frac{1}{1-x-yx z}$, which turns out to equal $N(x,y)$.
\end{proof}
We note that the continued fraction
$$\cfrac{1}{1-\cfrac{yx}{1+(y-1)x-\cfrac{yx}{1+(y-1)x-\cfrac{yx}{1+(y-1)x-\cdots}}}}$$ expands to give the version of the Narayana numbers that begins
$$\left(
\begin{array}{cccccc}
 1 & 0 & 0 & 0 & 0 & 0 \\
 0 & 1 & 0 & 0 & 0 & 0 \\
 0 & 1 & 1 & 0 & 0 & 0 \\
 0 & 1 & 3 & 1 & 0 & 0 \\
 0 & 1 & 6 & 6 & 1 & 0 \\
 0 & 1 & 10 & 20 & 10 & 1 \\
\end{array}
\right).$$

The number triangle $\mathbf{N}\mathbf{B}$ will have its generating function given by the Jacobi continued fraction
$$\cfrac{1}{1-(y+2)x-\cfrac{(y+1)x^2}{1-(y+2)x-\cfrac{(y+1)x^2}{1-(y+2)x-\cdots}}}.$$
This triangle begins
$$\left(
\begin{array}{cccccc}
 1 & 0 & 0 & 0 & 0 & 0 \\
 2 & 1 & 0 & 0 & 0 & 0 \\
 5 & 5 & 1 & 0 & 0 & 0 \\
 14 & 21 & 9 & 1 & 0 & 0 \\
 42 & 84 & 56 & 14 & 1 & 0 \\
 132 & 330 & 300 & 120 & 20 & 1 \\
\end{array}
\right).$$

The $(n,k)$-th element of this array gives the number of Schr\"oder paths of semi-length $n$ containing exactly $k$ peaks but no peaks at level one (\seqnum{A126216}). We have the following result.

\begin{proposition} The generating function of the product $\mathbf{N} \mathbf{B}$ can be expressed as the following Thron continued fraction.
$$\cfrac{1}{1-x-\cfrac{(y+1)x}{1+yx-\cfrac{(y+1)x}{1+yx - \cdots}}}.$$
\end{proposition}
\begin{proof} We solve the equation
$$u=\frac{1}{1+yx-(y+1)x u}$$ to obtain the generating function in the form
$$\frac{1}{1-x-(y+1)x u}.$$
Simplifying, we can show that this is the same as the generating function obtained from the Jacobi expression for the generating function of the triangle.
\end{proof}
Related to this is the following result.
\begin{proposition} The triangle that begins
$$\left(
\begin{array}{cccccc}
 1 & 0 & 0 & 0 & 0 & 0 \\
 1 & 0 & 0 & 0 & 0 & 0 \\
 2 & 1 & 0 & 0 & 0 & 0 \\
 5 & 5 & 1 & 0 & 0 & 0 \\
 14 & 21 & 9 & 1 & 0 & 0 \\
 42 & 84 & 56 & 14 & 1 & 0 \\
\end{array}
\right),$$  which counts little $q$-Schr\"oder paths, has a generating function given by each of the following three equivalent continued fractions.
$$\cfrac{1}{1-\cfrac{x}{1-\cfrac{(y+1)x}{1-\cfrac{x}{1-\cfrac{(y+1)x}{1-\cfrac{x}{1-\cdots}}}}}}.$$
$$\cfrac{1}{1-\cfrac{x}{1-yx-\cfrac{x}{1-yx-\cfrac{x}{1-yx-\cdots}}}}.$$
$$\cfrac{1}{1-x-\cfrac{(y+1)x^2}{1-(y+2)x-\cfrac{(y+1)x^2}{1-(y+2)x-\cdots}}}.$$
\end{proposition}
Thus we have Stieltjes, Thron and Jacobi continued fractions expressing the generating functions of this triangle. We note that Yang and Yang \cite{YangYang} have provided a bijection between $(q+2,q+1)$-Motzkin paths of length $n$ and small $q$-Schr\"oder paths of semi-length $n+1$ corresponding to the above two propositions.

Now applying $\mathbf{B}^{-1}$ on the left, we form the number triangle
$$ \mathbf{B}^{-1} \mathbf{N} \mathbf{B}$$ to obtain the number triangle whose generating function is given by the following Jacobi continued fraction.
$$\cfrac{1}{1-(y+1)x-\cfrac{(y+1)x^2}{1-(y+1)x-\cfrac{(y+1)x^2}{1-(y+1)x-\cdots}}}.$$ This binomial conjugate of the Narayana triangle begins
$$\left(
\begin{array}{cccccc}
 1 & 0 & 0 & 0 & 0 & 0 \\
 1 & 1 & 0 & 0 & 0 & 0 \\
 2 & 3 & 1 & 0 & 0 & 0 \\
 4 & 9 & 6 & 1 & 0 & 0 \\
 9 & 26 & 26 & 10 & 1 & 0 \\
 21 & 75 & 100 & 60 & 15 & 1 \\
\end{array}
\right).$$ This is \seqnum{A177896}. The elements of the first column are the Motzkin numbers, and its row sums \seqnum{A071356}, which begin
$$1, 2, 6, 20, 72, 272, 1064, 4272, 17504, 72896,\ldots$$ count Motzkin paths in which both the horizontals and the rises have $2$ colors. The diagonal sums \seqnum{A078481} of this matrix begin
$$1, 1, 3, 7, 19, 53, 153, 453, 1367, 4191, 13015,\ldots.$$
It follows that this sequence has a generating function given by the continued fraction
$$\cfrac{1}{1-(1+x)x-\cfrac{(1+x)x^2}{1-(1+x)x-\cfrac{(1+x)x^2}{1-(1+x)x-\cdots}}}.$$
This generating function also has a Thron continued fraction expression, given by
$$\cfrac{1}{1+x-\cfrac{2x}{1+x-\cfrac{x}{1+x-\cfrac{2x}{1+x-\cdots}}}}.$$
This shows it to be the transform of the large Schr\"oder numbers $S_n$ \cite{Trans} given by
$$\sum_{k=0}^n \binom{n+k}{2k}(-1)^{n-k} S_k.$$ The sequence counts  Dyck paths of semi-length $n$ with no $UDUD$. Equivalently, it counts Schr\"oder paths of length $2n$ whose rises alternate between having $2$ and $1$ colors, and whose horizontals have a weight of $-1$.

\section{Conclusions} We have shown that certain Riordan arrays have generating functions that are expressible as Jacobi or Thron continued fractions. The feature of these Riordan arrays is that they are related to counting lattice paths whose steps at level $0$ are privileged. This effect is conveyed through the second variable of the bivariate generating function. In contrast, other non-Riordan number triangles such as the Narayana triangle, also associated to lattice paths, carry the dependence on the second variable beyond level $0$.

\bigskip
\hrule

\noindent 2010 {\it Mathematics Subject Classification}:
Primary 15B36; Secondary 05A15, 11B83, 11J70, 11A55, 42C05.
\noindent \emph{Keywords:} Riordan array, Lattice path, Stieltjes continued fraction, Jacobi continued fraction, Thron continued fraction, Narayana numbers, generating function, orthogonal polynomial, Laurent biorthogonal polynomial

\begin{thebibliography}{99}


\bibitem{Inv} P. Barry, Chebyshev moments and Riordan involutions, Preprint (2019), \url{https://arxiv.org/abs/1912.11845}.
    
\bibitem{LBP} P. Barry, Constant coefficient Laurent biorthogonal
polynomials, Riordan arrays and moment
sequences, Preprint (2019), \url{https://arxiv.org/abs/1906.06370}.

\bibitem{OP} P. Barry and A. M. Mwafise, Classical and semi-classical orthogonal polynomials defined by Riordan arrays, and their moment sequences, \emph{J. Integer Seq.}, \textbf{21} (2018), Article 18.1.5.
    
\bibitem{Book} P. Barry, \emph{Riordan Arrays: a Primer}, Logic Press, 2017.

\bibitem{Prod} P. Barry, Constructing exponential Riordan arrays from their $A$ and $Z$ sequences, \emph{J. Integer Seq.}, \textbf{17} (2014), Article 14.2.6.

\bibitem{Nara} P. Barry and A. Hennessy, A note on Narayana triangles and related
polynomials, Riordan arrays, and MIMO
capacity calculations, \emph{J. Integer Seq.}, \textbf{14} (2011), Article 11.3.8.

\bibitem{Trans} P. Barry, Continued fractions and transformations of integer sequences, \emph{J. Integer Seq.}, \textbf{12} (2009), Article 09.7.6.

\bibitem{Chang} X.-K. Chang, X.-B. Hu, H. Lei, and Y.-N. Yeh, Combinatorial proofs of addition formulas, \emph{Electron. J. Combin.}, \textbf{23} (2016), \#P1.8.

\bibitem{Flajolet} P. Flajolet, Combinatorial aspects of continued fractions, \emph{Discrete Math.}, \textbf{32}
(1980), 125–-161.

\bibitem{Josuat} M. Josuat-Verg\`es, A $q$-analog of Schl\"afli and Gould identities on Stirling numbers, \emph{Ramanujan J.}, \textbf{46} (2018), 483–-507.

\bibitem{Oste} R. Oste and J. Van der Jeugt, Motzkin paths, Motzkin polynomials and recurrence relations, \emph{Electron. J. Combin.}, \textbf{22} (2015), \#P2.8.

\bibitem{Sokal} M. P\'etr\'eolle, A. D. Sokal, B-X. Zhu, Lattice paths and branched continued fractions:
An infinite sequence of generalizations of the Stieltjes--Rogers and Thron--Rogers polynomials,
with coefficientwise Hankel-total positivity, Preprint (2018), \url{https://arxiv.org/pdf/1807.03271.pdf}.

\bibitem{Sawa1} K. Sawa and Y. Nakamura, Application of the Lanczos-Phillips algorithm to continued fractions and its extension with orthogonal polynomials,  \emph{JSIAM Lett.}, \textbf{10} (2018), 57--60.

\bibitem{Sawa2} K. Sawa and Y. Nakamura, Extension of the Lanczos-Phillips algorithm with Laurent biorthogonal polynomials and its application to the Thron continued fractions, \emph{JSIAM Lett.}, \textbf{12} (2020), 1--4.

\bibitem{SGWW} L. W. Shapiro, S. Getu, W-J. Woan, and L.C. Woodson,
The Riordan group, \emph{Discr. Appl. Math.}, \textbf{34} (1991),
 229--239.

\bibitem{SL1} N. J. A.~Sloane, \emph{The
On-Line Encyclopedia of Integer Sequences}. Published electronically
at \texttt{http://oeis.org}, 2021.

\bibitem{SL2} N. J. A.~Sloane, The On-Line Encyclopedia of Integer
Sequences, \emph{Notices Amer. Math. Soc.}, \textbf{50} (2003),  912--915.

\bibitem{Viennot} X. Viennot, Introduction to chapter $3$ on continued fractions, Preprint (2013), \url{https://www.stat.purdue.edu/~mdw/ChapterIntroductions/ContinuedFractionsUpdateViennot.pdf}.

\bibitem{YangYang} L. Yang and S.-L. Yang, A relation between Schr\"oder paths and Motzkin paths, \emph{Graphs Combin.}, \textbf{36} (2020), 1489--1502.

\end{thebibliography}
\end{document}